\documentclass[11pt,reqno]{amsart}

\usepackage{amsmath, amsfonts, amssymb, amscd, enumerate}
\theoremstyle{plain}
\newtheorem{theo}{Theorem}[section]
\newtheorem{lem}[theo]{Lemma}

\theoremstyle{definition}

\newenvironment{proofwithnoend}{\noindent{\it Proof. }}{\par\medskip}

\theoremstyle{plain}
\newtheorem{lemma}[theo]{Lemma}

\theoremstyle{definition}

\newcommand{\beq}{\begin{equation}}
\newcommand{\eeq}{\end{equation}}

\renewcommand{\d}{\delta}

\newcommand{\g}{\gamma}

\renewcommand{\l}{\lambda}

\newcommand{\bC}{\mathbb{C}}
\newcommand{\bR}{\mathbb{R}}

\newcommand{\ga}{\mathfrak{a}}

\renewcommand{\gg}{\mathfrak{g}}
\newcommand{\gh}{\mathfrak{h}}
\newcommand{\gk}{\mathfrak{k}}

\newcommand{\gm}{\mathfrak{m}}
\newcommand{\gn}{\mathfrak{n}}
\newcommand{\go}{\mathfrak{o}}

\newcommand{\gt}{\mathfrak{t}}
\newcommand{\gu}{\mathfrak{u}}

\newcommand{\gz}{\mathfrak{z}}

\newcommand{\gX}{\mathfrak{X}}

\newcommand{\so}{\mathfrak{so}}
\newcommand{\su}{\mathfrak{su}}

\newcommand\SO{\mathrm{SO}}
\newcommand\SU{\mathrm{SU}}
\newcommand\U{\mathrm{U}}
\newcommand\T{\mathrm{T}}
\newcommand\Spin{\mathrm{Spin}}
\newcommand\Sp{\mathrm{Sp}}

\newcommand\G{\mathrm{G}}

\newcommand{\cB}{\mathcal{B}}
\newcommand{\cC}{\mathcal{C}}

\newcommand{\cO}{\mathcal{O}}

\newcommand\na{{\nabla}}
\newcommand\hv{\widehat v}
\newcommand\hw{\widehat w}
\newcommand\hA{\widehat A}

\renewcommand{\square}{\kern1pt\vbox
{\hrule height 0.6pt\hbox{\vrule width 0.6pt\hskip 3pt
\vbox{\vskip 6pt}\hskip 3pt\vrule width 0.6pt}\hrule height0.6pt}\kern1pt}

\DeclareMathOperator\End{End\;}

\DeclareMathOperator\Ad{Ad}
\DeclareMathOperator\ad{ad}

\newcommand{\wt}{\widetilde}
\newcommand{\wh}{\widehat}

\newcommand{\be}{\begin{equation}}
\newcommand{\ee}{\end{equation}}

\def\<#1,#2>{\langle\,#1,\,#2\,\rangle}
\newcommand{\arr}{\begin{array}{rlll}}
\newcommand{\ea}{\end{array}}
\newcommand{\bea}{\begin{eqnarray}}
\newcommand{\eea}{\end{eqnarray}}
\newcommand{\bean}{\begin{eqnarray*}}
\newcommand{\eean}{\end{eqnarray*}}

\def\sideremark#1{\ifvmode\leavevmode\fi\vadjust{
\vbox to0pt{\hbox to 0pt{\hskip\hsize\hskip1em
\vbox{\hsize3cm\tiny\raggedright\pretolerance10000
\noindent #1\hfill}\hss}\vbox to8pt{\vfil}\vss}}}

\newcounter{ssig}
\newcounter{ttig}
\begin{document}
%
%
\title[$6$-dimensional nearly K\"ahler manifolds of cohomogeneity one]{$6$-dimensional
nearly K\"ahler manifolds\\ of cohomogeneity one}
\author{Fabio Podest\`a and Andrea Spiro}
\begin{abstract} We consider $6$-dimensional strict nearly K\"ahler manifolds acted on by a compact, cohomogeneity
one automorphism group $G$. We classify the compact manifolds of this class up to $G$-diffeomorphisms.
We also prove that the manifold has constant sectional curvature whenever the group $G$ is simple. \end{abstract}
\maketitle

\section{Introduction}
\setcounter{equation}{0}
\bigskip
A Riemannian manifold ($M,g$) is called {\it nearly K\"ahler} (shortly NK) if it admits an almost complex structure $J$ such that $g$ is Hermitian and $(\nabla_XJ)X = 0$ for every vector field $X$. A NK manifold ($M,g,J$) is called {\it strict} if $\nabla_v J|_p\neq 0$ for every $p\in M$ and every $0\neq v\in T_pM$.
These manifolds have been investigated by A. Gray and many others (see e.g. \cite{G1,G2,M,N1,N2}). The class of NK manifolds includes $3$-symmetric spaces with canonical almost complex structure and it is one of the
sixteen Gray and Hervella' s classes of almost Hermitian manifolds (\cite{G4,GH}).
Moreover their canonical Hermitian connection $D$ has totally skew and $D$-parallel torsion (see e.g. \cite{A,N1}).\par
Recently Nagy proved the following two structure results: a) any NK manifold is locally holomorphically
isometric to the product of a K\"ahler manifold and a strict NK manifold; b) any complete strict NK manifold
is finitely covered by a product of homogeneous $3$-symmetric manifolds, twistor spaces of positive quaternion
K\"ahler manifolds with their canonical NK structure and $6$-dimensional strict NK manifolds (\cite{N1,N2}).\par
Therefore $6$-dimensional strict NK manifolds appear to be the interesting manifolds to focus on. Notice also that
for any  complete $6$-dimensional strict NK manifold ($M^6,g,J$) the structure group reduces to $\SU_3$ and the
first Chern class $c_1(M,J)$ vanishes. In addition, the strictness condition is equivalent to being non-K\"ahler and
the metric $g$ is automatically Einstein with positive scalar curvature. Finally, on simply connected
$6$-dimensional compact manifolds there is a bijective
correspondence between NK structures and unit real Killing spinors
(\cite{Gr}).\par
 All known examples of $6$-dimensional complete, strict NK  manifolds are homogeneous, namely the standard sphere
 $S^6 = \G_2/\SU_3$, the twistor spaces $\bC P^3 = \Sp_2/\T^1\times \Sp_1$ and $F_{1,2} = \SU_3/\T^2$ and the
 homogeneous space $S^3\times  S^3 = \SU_2^3/(\SU_2)_{\operatorname{diag}}$. Actually, Butruille  proved that these
 space exhaust the class of homogeneous strict NK manifolds in six dimensions (\cite{B}).\par\medskip
 In this paper we address the problem of classifying $6$-dimensional strict NK manifolds admitting a
 compact automorphism group acting by cohomogeneity one.
Our first result concerns the classification up to a $G$-diffeomorphism of $6$-dimensional
compact strict (i.e. non-K\"ahler)  NK manifolds of
cohomogeneity one . \par
\begin{theo}\label{T1}
 Let ($M,g,J)$ be a $6$-dimensional simply connected, compact   non-K\"ahler NK manifold and  $G$ a compact connected
 Lie group of isometries of ($M,g$) acting by cohomogeneity one. Then, only one of the following cases may occur:
\begin{itemize}
\item[a)] $G$ is  $\SU_3$ or $\SO_4$ and $M$ is $G$-diffeomorphic to $S^6$;
\item[b)] $G$ is  $\SU_2 \times \SU_2$ and $M$ is $G$-diffeomorphic to $\bC P^2$ or $S^3 \times S^3$.
\end{itemize}
\end{theo}
This result  motivates the study  of cohomogeneity one $6$-dimensional non-K\"ahler NK manifolds
acted on by the groups $\SU_3$ and $\SU_2  \times \SU_2$. Our next main result concerns
 $6$-dimensional non-K\"ahler NK-manifolds, not necessarily complete,
 acted on by a simple group $G$ with cohomogeneity one.  \par
\begin{theo}\label{T2}
Let $(M,g,J$) be a 6-dimensional  non-K\"ahler NK  manifold. If a simple Lie group $G$ acts on $M$ as a
group of automorphisms of the NK structure with cohomogeneity one, then $(M,g)$ has constant sectional
curvature  and $G = \SU_3$.
\end{theo}
We will discuss the cohomogeneity one 6-dimensional non-K\"ahler NK manifolds acted on by $\SU_2 \times \SU_2$ in a
future paper.\par\medskip
The structure of the paper is as follows. In Section 2, we give some preliminaries concerning nearly K\"ahler
structures and cohomogeneity one Riemannian manifolds. Sections 3 and 4 are devoted to the proof of
Theorem \ref{T1} and \ref{T2} respectively.  \par
\medskip
Given a Lie group $G$, the corresponding Lie algebra $Lie(G)$ will be always denoted by the corresponding gothic
letter $\gg$. We will also indicate by $\cB$ the opposite of the Cartan-Killing form of a semisimple $\gg$.
Finally,  given an action of  a Lie group $G$ on a manifold $M$ and an element $X \in Lie(G)$, the symbol $\wh X$
will  denote the vector field on $M$ corresponding to $X$.\par
\bigskip
\bigskip
\section{Preliminaries}\label{prelim}
\setcounter{equation}{0}
\subsection{Cohomogeneity one actions} \label{prelim1} Let $G$ be a compact connected Lie group acting almost effectively and isometrically on a Riemannian manifold ($M,g$). The action is called of {\it cohomogeneity one} if generic orbits have codimension one or, equivalently, the orbit space $M/G$ is one dimensional,
possibly with boundary (see e.g. \cite{AA}, \cite{GWZ}, \cite{U} for a detailed exposition).\par
When $M$ is compact with finite fundamental group, the orbit space $M/G$ has exactly two boundary points which correspond to two singular orbits. Denote by $\pi:M\to M/G$ the standard projection and fix a normal geodesic $\gamma:[a,b]\to M$, i.e. a geodesic that meets every $G$-orbit orthogonally, so that $\pi\circ\gamma$ is a homeomorphism between $[a,b]\subset \bR$ and the orbit space $M/G$. The stabilizers $G_{\gamma(t)}$ with $ {t\in ]a,b[}$ are all
equal to a principal isotropy subgroup $K$, while $H_1 = G_{\gamma(a)}$ and $H_2 = G_{\gamma(b)}$ are two singular isotropies both containing $K$.\par
The normal isotropy representation $\nu_i$ maps $H_i$ onto a linear subgroup $\wh H_i\subseteq \operatorname{O}(k_i)$, where $k_i = \operatorname{codim} G/H_i$, which acts transitively on the unit sphere $S^{k_i-1}\cong \wh H_i/ \wh K$, where $\wh K = K/\ker \nu_i$.  In the following table, we reproduce  the well-known  Borel's list of the pairs $(\wh H, \wh K)$ of  connected compact linear groups acting  transitively on spheres.\par
\vskip 0.5 cm
\centerline{
\tiny
\vbox{\offinterlineskip
\halign {\strut
\vrule \vrule \vrule\hfil\  $#$\  \hfil
&\vrule  \vrule\vrule\hfil\  $#$ \ \hfil
&\vrule \vrule\hfil\  $#$\ \hfil
& \vrule \vrule\hfil\  $#$\ \hfil
& \vrule \vrule \hfil\  $#$\ \hfil
& \vrule \vrule\hfil\  $#$\ \hfil
&\vrule \vrule \hfil\  $#$\ \hfil
\vrule \vrule
 \cr
\noalign{\hrule \hrule}
\underset{\phantom{A}}{\overset{\phantom{A}}{\wh H}} &
\SO_n &
 \U_n\hfil\vrule\hfil\SU_n &
\Sp_n \Sp_1\hfil\vrule\hfil \Sp_n \U_1\hfil\vrule\hfil\Sp_n  & \operatorname{G}_2 & \Spin_7 & \Spin_9\cr
\noalign{\hrule\hrule}
\underset{\phantom{A}}{\overset{\phantom{A}}{\wh K}}  & \SO_{n-1} & \U_{n-1}\ \  \hfil\vrule\hfil\ \  \SU_{n-1} & \ \Sp_{n-1} \Sp_1\ \ \vrule\hfil\ \ \Sp_{n-1} \U_1\ \vrule\hfil\ \ \  \Sp_{n-1}
 & \SU_3 & \operatorname{G}_2 & \Spin_7 \cr\noalign{\hrule\hrule}
\underset{\phantom{A}}{\overset{\phantom{A}}{\wh H/\wh K}} & S^{n-1} & S^{2n-1} & S^{4n-1}
 & S^6 & S^7 & S^{15} \cr\noalign{\hrule\hrule}
}}
}
\vskip - 0.1 cm
\centerline{\tiny\bf Table 1}
\medskip

Therefore the triple ($H_1,K,H_2$) is a so called {\it admissible triple} of subgroups, i.e. three subgroups
such that $K\subseteq H_1$, $H_2$ and so that $H_i/K$, $i=1,2$, is diffeomorphic to a sphere. Conversely, for any admissible triple of subgroups of $G$ one can construct a $G$-manifold $M$ of cohomogeneity one with singular
and principal isotropy subgroups equal to the given triple (see e.g. \cite{GWZ}). Two $G$-manifolds of cohomogeneity one are $G$-diffeomorphic if and only if their associated triples are obtained one from the other by a suitable combination of the following operations: \begin{itemize}
\item[i)] switch $H_1$ and $H_2$;
\item[ii)] conjugate every subgroup $K,H_1,H_2$ by the same element of $G$;
\item[iii)] conjugate $H_1$ by an element of the normalizer $N_G(K)^o$.
\end{itemize}
\par
\smallskip
Let us  denote by $\xi$ a unit vector field defined on the open subset $M_{\operatorname{reg}}\subset M$ of regular points orthogonal to all $G$-orbits; Its integral curves are normal geodesics.
Any regular orbit $G\cdot p = G/K$ admits a tubular neighborhood which is $G$-equivariantly isometric to
the product $]a,b[\times G/K$, $]a,b[\subset \bR$, endowed with the metric
$$\bar g = dt^2 + g_t,$$
where $g_t$ is a smooth family of $G$-invariant metrics on $G/K$ and the vector field $\partial/\partial t$ corresponds to $\xi$. The metrics $g_t$ can be described as follows:
fix an $\Ad_K$-invariant decomposition $\gg = \gk + \gm$ so that, for any $t$, $g_t$ can be identified with an $\Ad_k$-invariant
scalar product on $\gm$. Therefore the family $g_t$ can be seen as a smooth map $g_t:]a,b[\to S^2_+(\gm)^K$ into the space $S^2_+(\gm)^K$ of $\Ad_K$-invariant symmetric positively defined bilinear maps on $\gm$.\par\medskip
\subsection{6-dimensional NK manifolds} \label{prelim2}
Given a NK manifold ($M,g,J$),  we  denote by $\omega = g(J\cdot, \cdot)$ the associated K\"ahler form.

We recall that
\beq
d\omega(X,Y,Z) = 3 g((\nabla_X J)Y,Z),\qquad X,Y,Z\in \gX(M). \eeq
Let us focus on the six dimensional non-K\"ahler NK manifolds. We recall that these manifolds are strict and Einstein with positive scalar curvature. In particular they have finite fundamental group whenever complete.
Moreover, for their  group of isometries the following holds (\cite{M} Prop 3.1).
\begin{lemma}\label{spinor} If $G$ is a connected group of isometries  of a compact, non-K\"ahler NK manifold $(M,g,J)$ of dimension six, then $G$ preserves $J$ unless $(M,g)$ is isometric to the standard sphere $S^6$.\end{lemma}
The following lemma is a generalization of the result in \cite{G3}.
\begin{lemma}\label{subm} A 6-dimensional compact,
non-K\"ahler,  NK  manifold ($M,g,J$) admits no almost complex four-dimensional submanifold.\end{lemma}
\begin{proofwithnoend} Let $N\subset M$ be a $4$-dimensional almost complex submanifold. In \cite{G3}, Lemmas 2, 3, it is proved that $N$, endowed with the induced Hermitian structure, is K\"ahler. This implies that for any $X,Y\in \gX(N)$
\beq\label{nabla}(\nabla_XJ)Y = h(X,JY) - Jh(X,Y),\eeq
and
\beq\label{asym}h(X,JY) + h(JX,Y) = 2 Jh(X,Y)\eeq
where $h$ denotes the second fundamental form. We select a point $p\in N$ and a unit vector $v\in T_pN$ such that
$$||h(v,v)||^2 = \max_{w\in T_pN;\ ||w||=1}||h(w,w)||^2\ .$$ This implies that for every $w\in T_pN$ with $g(v,w) = 0 $ we have $g(h(v,v),h(v,w)) = 0$. Since we can also use $Jv$ instead of $v$, for every $w\in T_pN$ with $g(Jv,w)=0$, we have $g(h(Jv,Jv),h(Jv,w)) = 0$. Using now  \eqref{asym}, for any $w\in T_pN$ with $g(v,w) = g(Jv,w) = 0$, we have
$$2 g(Jh(v,v),h(v,w)) = - g(h(v,Jw),h(v,v)) - g(h(Jv,w), h(v,v)) = $$
$$ = - g(h(v,Jw),h(v,v)) + g(h(Jv,w), h(Jv,Jv)) = 0.$$
Hence the normal vector $h(v,w)$ is orthogonal to $h(v,v)$ and $Jh(v,v)$. These two vectors span the normal space
$T_pN^\perp$ if $h$ is not zero, hence $h(v,w) = 0$. Similarly $h(v,Jw) = 0$.
 So, by  \eqref{nabla}, we have the existence of unit tangent vectors $v,w$ with $g(v,w) = g(v,Jw) = 0$ and $(\nabla_vJ)w = 0$.
On the other hand, this is impossible, as 6-dimensional strict NK manifolds are of constant type (\cite{G2}, Thm. 5.2), i.e. there is a positive constant $\alpha\in \bR$ such that for every $x,y\in T_q(M)$, $q\in M$,
$$||(\nabla_xJ)y||^2 = \alpha\left(||x||^2||y||^2 - g(x,y)^2 - g(Jx,y)^2\right).\eqno{\qed}
$$\end{proofwithnoend}
\par
\medskip
  We conclude  describing  known examples of cohomogeneity one actions on 6-dimensional non-K\"ahler NK manifolds. These are indeed the cases considered  in  Theorem \ref{T1}. \par
  \begin{itemize}
  \item[i)] The sphere  $S^6$ has a $\G_2$-invariant NK structure and the subgroup $G = \SO_4\subset \G_2$ acts on
  $S^6 = \G_2/\SU_3$ by cohomogeneity one. Using the fact that  the corresponding action
  of $\SO_4$ on $\bR^7 = \bR^3 + \bR^4$ acts on the summands $\bR^3$ and $\bR^4$ with
  the standard $\SO_3$- and $\SO_4$-representations,  one can see  that the associated
  triple of this cohomogeneity one action on $S^6$  is ($\SO_3,\T^1,\U_2$).
  The subgroup $\SU_3\subset \G_2$ also acts on $S^6$ by cohomogeneity one with
  two isolated fixed points and associated triple ($\SU_3,\SU_2,\SU_3$).\par
\item[ii)] The complex projective space $\bC P^3$ has a strict NK structure, invariant by the
standard $\Sp_2$-action on it. One can immediately check that the diagonal subgroup
$G = \Sp_1\times \Sp_1 \subset \Sp_2$ acts on $\bC P^3$ by cohomogeneity one with associated triple  ($\T^1\times\Sp_1,\T^1_{\operatorname{diag}},\Sp_1\times\T^1$).\par
\item[iii)]The $3$-symmetric space $S^3\times S^3$ is homogeneous under the following action of  $\SU_2^3$: for any $(g_1,g_2,g_3)\in \SU_2^3$ and $(x_1,x_2)\in \SU_2^2\cong S^3 \times S^3$
$$(g_1,g_2,g_3)\cdot(x_1,x_2) = (g_1x_1g_3^{-1}, g_2x_2g_3^{-1}).$$
The group $G =\SU_2^2$,  embedded in $\SU^3_2$ by the homomorphism  $ (g,h) \in \SU_2^2\longmapsto (g,h,g)\in \SU_2^3$,  acts on $S^3 \times S^3$ by cohomogeneity one with associated triple  ($(\SU_2)_{\operatorname{diag}},\T^1_{\operatorname{diag}},(\SU_2)_{\operatorname{diag}}$).\par
\end{itemize}
\par
\bigskip
\section{ Proof of Theorem \ref{T1}}
\label{PT1}
\setcounter{equation}{0}
In all the following, $(M, g, J)$ is a simply connected, compact, 6-dimensional, non-K\"ahler NK manifold and $G$ acts  almost effectively as a group of isometries of $(M,g)$ with principal orbit of codimension 1. By Lemma \ref{spinor} we can suppose that  $G$ preserves the almost complex structure $J$. We fix a principal point $p\in M$ and we denote $K = G_p$. We start with the following:
\begin{lem} \label{provina} The subgroup $K$ is isomorphic to a compact subgroup of $\SU_2$ and hence $\gk$ is $\{0\}$, $\bR$ or $\su_2$.
\end{lem}
\begin{proof} Since the NK structure is non-K\"ahler, $d\omega_p\not=0$ and therefore $K$, that preserves
$g_p$, $J_p$ and $d \omega_p$, must be a subgroup of $\SU_3$. Since it preserves also the  unit vector $\xi_p$,  normal to the orbit $G \cdot p$,   and the vector  $J\xi_p$,  it follows that  $K \subseteq \SU_2$.
\end{proof}
From Lemma \ref{provina} and the fact  $\dim \gg - \dim \gk = \dim G \cdot p = 5$ for any principal point $p$,  we may easily determine the  possibilities for the compact Lie algebra $\gg$, namely:
\begin{itemize}
\item[--] if $\gk = \{0\}$, then $\gg = \su_2+ 2\mathbb R$ or  $5\mathbb R$;
\item[--] if $\gk = \mathbb R$, then $\gg = \su_2+\su_2$, $\su_2+3\mathbb R$ or $6\mathbb R$;
\item[--] if $\gk = \su_2$, then \ $\gg = 2\su_2 + 2\mathbb R$,  $\su_2+5\mathbb R$ or $\su_3$.
\end{itemize}
On the other hand, some of these possibilities can be immediately excluded.
First of all, the pairs  ($\gg$,$\gk$)=($6 \mathbb R, \mathbb R$) and  ($\su_2+5\mathbb R$, $\su_2$) cannot occur, because in those cases   $\gk$ would be an ideal of  $\gg$, the
principal isotropy would   act trivially on $M$ and the action would not be almost effective.  Secondly,
the pair ($\gg$,$\gk$)= ($\su_2+\su_2+2\mathbb R$, $\su_2$) is not admissible because in this case the fixed point set  $T_p(G\cdot p)^{K^o}$ is three dimensional and $J$-invariant, a contradiction. Finally, also the pair ($\gg$,$\gk$)=($5 \mathbb R, \{0\}$) can be ruled out.  In fact, since $\gg$ is abelian, the two singular orbits should be both diffeomorphic to  $\T^4$.  Since $M$ is the union of two disk bundles over the singular orbits,  Seifert-VanKampen theorem implies that $\pi_1(M)$ is infinite, a contradiction. In addition we have the following lemma.
\begin{lem} The pairs  $(\gg,\gk) = (\su_2+3\mathbb R,\bR)$ and $(\su_2+ 2\mathbb R,\{0\})$ are not admissible.
\end{lem}
\begin{proof} Let $(\gg,\gk)$ be $(\su_2+3\mathbb R,\bR)$. Note that the Lie algebra $\gk$ has a non trivial
projection onto the $\su_2$ factor because otherwise it would be contained in the center of $\gg$.
Let $S := G\cdot q$ be a singular orbit with isotropy subalgebra $\gg_q = \gh\supseteq \gk$. We denote by
$\nu:\gh\to \go(T_qS^\perp)$ the slice representation  and by
$\gk_o = \ker \nu|_{\gk}\subseteq \gk$.\par  The case $\gk_o = \gk = \mathbb R$ cannot occur by the following arguments.
Since $\gk = \gk_o$ is an ideal of $\gh$, using Borel's list we see that  $\gh/\gk$ is isomorphic to $\mathbb R$ or
$\su_2$. The latter  is excluded because $\gk$ would be in the centralizer of $\su_2$, hence in the center of $\gg$.
Therefore we have $\gh \cong 2\mathbb R$ and $\gh = \gk + \gh_o$ for some $\gh_o\subset \gz(\gg)$. Now, $\gh_o$ acts
non trivially on the normal space to $S$, while it acts trivially on the tangent space $T_qS$ because it is contained in
the center of $\gg$. Since the isotropy representation preserves $J$, the tangent space $T_qS = (T_qM)^{\gh_o}$ is $J$- invariant. This means that the orbit
$S$ is almost complex and four dimensional, which cannot be by Lemma \ref{subm}.\par
Therefore $\gk_o = \{0\}$ and from Borel's list, $(\gh,\gk) = (\so_3,\bR)$ or $(\gu_2,\bR)$.
This implies that both singular orbits are diffeomorphic to tori, contradicting
the fact that $M$ has finite fundamental group.\par
The case $(\gg,\gk) = (\su_2+ 2\mathbb R,\{0\})$ can be ruled out using similar arguments.\end{proof}
Therefore the possible pairs ($\gg,\gk$) reduce to
$$ (\gg,\gk) = (\su_2 + \su_2, \mathbb R ) \quad \operatorname{or}\quad (\gg,\gk) = (\su_3, \su_2).$$
We now focus on the case $\gg = \su_2 + \su_2$ and hence on  almost effective actions of $G = \SU_2 \times \SU_2$.
\begin{lem} \label{lemma33} If $(\gg,\gk) = (\su_2 + \su_2, \mathbb R )$, the Lie algebra $\gk$ is, up to an automorphism of $\gg$, diagonally embedded into a Cartan subalgebra of $\gg$.
\end{lem}
\begin{proof} We  first observe that $\gk\cong \bR$ projects non trivially into each factor $\su_2$. In fact if $\gk$ is contained in some
factor $\su_2$, then the fixed point set $M^{K^o}$ would be $4$-dimensional and almost complex
and this is not possible by Lemma \ref{subm}. Then there is a unique Cartan subalgebra $\gt = pr_1(\gk) \oplus pr_2(\gk)$ containing $\gk$, where $pr_i$ denote the projections onto the $\su_2$-factors. Consider the $\cB$-orthogonal decompositions $\gt = \gk +\ga$ and $\gg = \gt + \gn_1+ \gn_2$, with $\gn_i$ contained in $\su_2$ factors, so that $\gm = \ga + \gn_1 + \gn_2$. \par
Since the $\gk$-modules $\ga,\gn_1,\gn_2$ are mutually non equivalent, we have that $S^2(\gm)^{\Ad(K)}$ is generated by $\cB|_{\ga\times \ga}$ and $\cB|_{\gn_i\times \gn_i}$ for $i=1,2$.
 According to the remarks in \S \ref{prelim1},
the metric $g$ in a tubular neighborhood $]a,b[ \times G/K$ of the principal orbit $G\cdot p = G/K$ is of the form
$$g= dt^2 + f(t)^2 \cB|_{\ga\times \ga} + h_1(t)^2 \cB|_{\gn_1\times \gn_1} +
h_2(t)^2 \cB|_{\gn_2\times \gn_2},$$
where  $f,h_1,h_2$ are smooth functions of $t\in ]a,b[$.
We now claim that $\nabla J = 0$, contradicting the fact that the NK structure is non-K\"ahler. In fact, we observe that $\Lambda^3(T_pM)^{K}$ is $(\bR\cdot \xi + \wh\ga) \otimes [\Lambda^2(\wh\gn_1) +
\Lambda^2(\wh\gn_2)]$, where $\xi= \frac{\partial}{\partial t}$. Hence the only possibly non zero
component of the $3$-form $\psi = d\omega = 3g(\nabla J \cdot,\cdot)$ are $\psi(\xi,\wh\gn_i,\wh\gn_i)$ and $\psi(\wh\ga,\wh\gn_i,\wh\gn_i)$.\par
 We show that $\psi(\xi,\wh\gn_i,\wh\gn_i) = 0$ and $\psi(\wh\ga,\wh\gn_i,\wh\gn_i) = 0$, from which the conclusion will
 follow.
 To do this, we observe that $J\wh \gn_i|_{\gamma(t)}= \wh\gn_i|_{\gamma(t)}$ and the induced complex structure $J_t\in \End(\gn_i)^{\Ad(K)}$ is constant since $\gn_i$ is two-dimensional. Pick $v,w\in \gn_i$ with $J\wh v = \wh w$ along $\gamma(t)$. Then
$$\psi(\xi,\wh v,\wh w) = g(\nabla_{\xi} \wh w, \wh w) - g(J\nabla_{\xi}\wh v,\wh w) = $$
$$= g(\nabla_{\xi} \wh w, \wh w)   - g(\nabla_{\xi}\wh v,\wh v) = \frac{1}{2} \frac{d}{dt}\left(g(\wh w,\wh w)  - g(\wh v,\wh v)\right)= 0.$$
In order to prove that $\psi(\wh\ga,\wh\gn_i,\wh\gn_i) = 0$, it is enough to note that $\psi$ is $J$-invariant and $J\wh\ga = \bR\cdot\xi$. \end{proof}
By the previous lemma and Borel's list, we see that a singular isotropy subalgebra $\gh\supsetneq \gk =\bR$ is either $\mathbb R^2$, $(\su_{2})_{\operatorname{diag}}$ or $\mathbb R +\su_2$, up to conjugation.
We claim
 that the case $\gh= \bR^2$ cannot occur. Indeed, if $\gh = \bR^2$ we see that the normal space to the singular orbit $\cO$ is the fixed point set of $K^o$, hence it is $J$-invariant. This means that $\cO$ is almost complex and $4$-dimensional, contradicting Lemma \ref{subm}. \par
We now claim that the admissible triple $(H_1, K,H_2)$ of the cohomogeneity one action consists of  connected groups.  Indeed,  every singular orbit has codimension greater than two and this implies that both singular orbits are simply connected, since each of them is a deformation retract of the complement of the other one. This implies that the two singular isotropy subgroups $H_1, H_2$ are connected. Since the sphere $\dim H_i/K$ have dimension greater than one, $K$ is connected too. \par
Hence the only possibilities of the triple ($H_1$, $K$, $H_2$) are, up to equivalence,  \begin{itemize}
\item[--] $(\T^1\times \SU_2, \T^1_{\operatorname{diag}},\SU_2\times \T^1)$;
\item[--] $((\SU_2)_{\operatorname{diag}},\T^1_{\operatorname{diag}},\SU_2\times \T^1)$;
\item[--] $((\SU_2)_{\operatorname{diag}},\T^1_{\operatorname{diag}},(\SU_2)_{\operatorname{diag}})$,
\end{itemize}
which correspond to the examples $\bC P^3$, $S^6$ and $S^3\times S^3$ of $\SU_2\times\SU_2$-manifolds presented  in \S\ref{prelim2}.\par
\medskip
We now deal with the case $\gg = \su_3$ and hence with  almost effective actions of $G = \SU_3$. Looking at Borel's list we see that a singular isotropy subalgebra $\gh$ containing $\gk = \su_2$ can be either $\gh =\su_3$ or $\gh = \su_2 \oplus \mathbb R$. In the latter case, the normal space  of the
singular orbit  is the fixed point set of the isotropy representation of ${K^o}$ and the singular orbit would be an almost  complex $4$-dimensional submanifold, which is impossible. Therefore both singular isotropy subalgebras coincide with $\su_3 = \gg$ and the singular orbits are fixed points. This means that there is only one possible triple ($H_1$, $K$, $H_2$), corresponding to the $\SU_3$-manifold $S^6$. \par\bigskip

\section{Proof of Theorem \ref{T2}.}
\setcounter{equation}{0}
\bigskip
By Lemma \ref{provina} and following remarks,  the only compact simple group, acting by cohomogeneity one as group of automorphisms of
the NK structure of  6-dimensional non-K\"ahler NK manifold $(M, g, J)$,  is $G = \SU_3$ with
principal isotropy $K = \SU_2$. \par
Consider the $\cB$-orthogonal decomposition of $\gg = \su_3$ into irreducible $\gk$-moduli
$\gg = \gk + \ga + \gn$, where  $ \dim \ga = 1$ and  $\gn \cong \bC^2$ is  the standard representation space of $\SU_2$. Notice  that $[\ga,\gk] = 0$ and $[\ga,\gn]\subset \gn$.  Let also  $A\in \ga$ so that
$J_o = \ad(A)|_{\gn}: \gn \to \gn $  coincides with an $\Ad_K$-invariant complex structure on $\gn$.\par
 By the observations in \S \ref{prelim1},
the metric $g$ in a tubular neighborhood $I \times G/K$, $I = ]a,b[$,  of a principal orbit $G\cdot p = G/K$ is of the form
\beq \label{metric} g = dt^2 + h^2 \cB|_{\ga\times\ga} + f^2 \cB|_{\gn\times\gn},\eeq
where $h,f$ are positive functions in $C^\infty(I)$   and where  $\partial/\partial t$ is identified with the unit vector field $\xi$, orthogonal to the $G$-orbits, and the curve $(t, e K) \in I \times G/K$ is identified with the  normal geodesic $\g(t)$.\par
Recall that the space of $\Ad_K$-invariant complex structures on $\gn = \bC^2$ is a $2$-sphere, isomorphic to the unit sphere of the imaginary quaternions, i.e.
$$S = \left\{\ \sum_{i=1}^3 a_i J_i\ :\ \sum_i a_i^2 = 1\ \right\}$$
where $J_i$ are anticommuting complex structures with $J_1J_2J_3 = -Id$. With no loss of generality we may assume that
$J_1 = J_o$.
\par
\smallskip
From $K$-invariance of $J$ we immediately have the following.
\begin{itemize}
\item[a)] For any $t$, the vector $J\xi|_{\g(t)}$ is $K$-fixed in the tangent space $V_t = T_{\g(t)}ÿ(G\cdot \g(t))$ to the $G$-orbit $G \cdot \g(t)$,
so that $J\xi|_{\g(t)}$ lies in the linear span of
$\wh A_{\gamma(t)}$.  We set $J\xi|_{\g(t)} = u(t)\cdot  \wh A|_{\g(t)}$ for some positive $u\in C^\infty(I)$;
\item[b)] $J$ induces an $\Ad_K$-invariant complex structure in $\gn$, say $J_t$, such that
for every $v\in \gn$ we have $J\wh v = \widehat{J_tv}$. By previous remarks,
$$J_t = \sum_{i}^3 a_i(t) J_i,$$
for some $a_i\in C^\infty(I)$ with $\sum_{i=1}^3 a_i^2 = 1$.
\end{itemize}
\medskip
We want now to  determine the equations that must  be satisfied by   the functions $h$, $f$ and $a_i$'s, so that $(g, J)$ is a NK structure.
The next four lemmata are devoted to the evaluation  of all components of  $\na J$ at the points of the normal geodesic $\g(t)$, by means of  which we will immediately determine  the equations we are looking for. \par
\smallskip
\begin{lem}\label{L1} The vector field $J\xi$ is of the form  $J\xi = u \wh A$ with $u^2 = \frac{1}{4 h^2}$ and the operators $\na_{\cdot} \xi$ and $\na_{\cdot}  J \xi$ satisfy the following identities:\par
\noindent (1)\ $\na\xi|_{\wh\gn} = \frac{f'}{f} Id$.\par
\noindent (2)\ $\na_\xi\xi = \na_\xi J\xi = 0$, $\na_{J\xi}\xi = \frac{h'}{h} J\xi$ and $ \na_{J\xi}J\xi =  - \frac{h'}{h} \xi$.\par
It follows that
\beq \label{zeroes} (\na_\xi J)\xi =  (\na_\xi J)J\xi = (\na_{J\xi}J)J\xi = (\na_{J\xi}J)\xi = 0\  .\eeq
\end{lem}
\begin{proof}  The first claim follows immediately from previous remarks and from  $1 = ||J\xi||^2 = u^2 ||\wh A||^2 = 4 u^2 h^2$. To see (1), observe that $\na\xi|_{\wh\gn}$ can be identified with   an element of $S^2(\gn)^{\Ad_K}$ at any point $\g(t)$, which is,  by Schur's Lemma,  of the form $\l(t) Id_\gn$. So for every $v\in \gn$ we have $\na_{\wh v}\xi = \lambda {\wh v}$  for some $\l \in \cC^\infty(I)$ and the claim  follows from
$$\lambda f^2 \cB(v,v) = g(\na_{\wh v}\xi, \wh v) = g(\na_\xi \wh v,\wh v) = \frac{1}{2}(||\wh v||^2)' = ff' \cB(v,v).$$
To check  (2), notice that $\na_\xi \xi = 0$ since $\xi$ is the tangent vector to the normal geodesic.  Moreover,  $\na_\xi J\xi$ is $K$-invariant and hence it is in the span of $\xi$, $J\xi$. Since  $g(\na_\xi J\xi,\xi) = -g(J\xi,\na_\xi\xi)=0$, we get $\na_\xi J\xi = 0$. Finally, since also   $\na_{J\xi}\xi$ and $\na_{J\xi}J \xi$  are $K$-invariant, the  last equalities  follows from $J \xi = u \wh A$, $g(\na_{J\xi}\xi,J\xi)  = - g(\na_{J\xi} J\xi,\xi)$  and
$$ g(\na_{J\xi}\xi,J\xi) = u^2g(\na_{\wh A}\xi,\wh A) = u^2g(\na_\xi\wh A,\wh A) = \frac{u^2}{2}(||\wh A||^2)' = 4 u^2 hh' \ .$$
Equalities \eqref{zeroes} are  direct  consequence  of  (2).
\end{proof}
\begin{lem}\label{NN}
  At any point of the geodesic $\gamma(t)$ and for any
$v \in \gn$
\beq\label{nn}(\na_{\wh v}J)\wh v = \left(\frac{uh^2}{2f^2} a_1 + \frac{f'}{f}\right) ||\wh v||^2 J\xi\ .
\eeq\end{lem}
\begin{proof} First of all, we claim  that $(\na_{\wh v}J)\wh v|_{\g(t)}$ has trivial orthogonal projections along $\bR \xi|_{\g(t)}$ and $\wh \gn|_{\g(t)}$. Indeed,  using  Lemma \ref{L1} and the properties that  $[\wh \gg,\xi] = 0$,  $\wh v$ is a Killing vector field and $\ad_A|_{\gn} = J_1$, we have$$g(\na_{\wh v}J\wh v, \xi) = \wh v\cdot g(J\wh v, \xi) - g(J\wh v, \na_{\wh v}\xi) = 0,$$
$$g(J\na_{\wh v}\wh v, \xi) = g(\na_{J\xi}\wh v, \wh v) = \frac{1}{2} J\xi\cdot ||\wh v||^2 =
- \frac{1}{2} u\ \wh A\cdot ||\wh v||^2 = - u\ g([\wh A,\wh v], \wh v) = 0$$
and this implies $g((\na_{\wh v}J)\wh v,\xi)_{\gamma(t)} = 0$. On the other hand, since $[\gn,\gn]\subseteq \gk +\ga$, we also have that $g([\wh v,\wh \gn], \wh \gn)_{\gamma(t)} = 0$ and hence
$g(\na_{\wh v}\wh v, \wh \gn)_{\gamma(t)} = 0$.   From this, using Koszul's formula,  the fact that $\wh v,\wh w$ are Killing fields and  setting $z = J_t v$  so  that
$\wh z|_{\gamma(t)} = J\wh v|_{\gamma(t)}$, we have that for any $w \in \gn$
$$ g((\na_{\wh v}J)\wh v,\wh w)_{\g(t)} = g(\na_{\wh v}J\wh v,\wh w)_{\g(t)} +  g( \na_{\wh v}\wh v, J\wh w)_{\g(t)}  =  \frac{1}{2}\left(J\wh v\cdot g(\wh v,\wh w)\right)|_{\g(t)} = $$
$$ = \frac{1}{2}\left(\wh z\cdot g(\wh v,\wh w)\right)|_{\g(t)} =\frac{1}{2}  \left.\left(g([\wh z,\wh v],\wh w) + g(\wh v, [\wh z,\wh w])\right)\right|_{\g(t)} = 0\ .$$
Hence,  to obtain  formula \eqref{nn}, we only need to compute $g((\na_{\wh v}J)\wh v,J\xi)|_{\g(t)}$.
By Koszul's formula and the fact that $\wh v$ and $\wh A$ are Killing and $J$-preserving,
\begin{equation}\label{Jxi} g(\na_{\wh v}J\wh v,J\xi)_{\g(t)} = u\ g(\na_{\wh v}J\wh v, \wh A)_{\g(t)} = \frac{u}{2}\left( J\wh v\cdot g(\wh v,\wh A)+ g([
\wh v,\wh A],J\wh v)\right)_{\g(t)}\ .\end{equation}
On the other hand, recalling that $J \wh w|_{\g(t)} = \sum_i a_i \wh{J_i v}|_{\g(t)}$,
$$g([\wh v,\wh A],J\wh v)_{\g(t)}= - g(\widehat{[v,A]},J\wh v) _{\g(t)}= g(\wh{J_1 v} , \sum_i a_i \wh{J_i w})_{\g(t)} = a_1 \left.||\wh v||^2\right|_{\g(t)}\ ,$$
while,  setting $z = J_t v =  \sum_i a_i J_i v$  so  that
$\wh z|_{\gamma(t)} = J\wh v|_{\gamma(t)} $,
$$\left(J\wh v\cdot  g(\wh v,\wh A)\right)_{\g(t)}   = \left(\wh z\cdot  g(\wh v,\wh A)\right)_{\g(t)}   = - g(\wh{[z,v]},\wh A)_{\g(t)}  -  g(\wh v, \wh{[z, A]})_{\g(t)}  = $$
$$ =
h^2 \cB (z, [A,v])  +  f^2 \cB(v, [A,z]) = $$
$$ = h^2 \cB (\sum a_i J_i v, J_1 v)  -  f^2 \cB(J_1v, \sum a_i J_i v) =  \left.\left(\frac{h^2}{f^2}-1\right) a_1 ||\wh v||^2 \right|_{\g(t)}\ .$$
From this and \eqref{Jxi}, we get
\beq \label{prima}Êg(\na_{\wh v}J\wh v,J\xi) = \frac{uh^2}{2f^2} a_1 ||\wh v||^2\ . \eeq
On the other hand, using Koszul's formula once again,
\begin{equation}\label{second}\left.g(J\na_{\wh v}\wh v, J\xi)\right|_{\g(t)}  = -\left.\frac{1}{2}(||\wh v||^2)' \right|_{\g(t)} = -ff'\cB(v,v) = - \left.\frac{f'}{f} ||\wh v||^2\right|_{\g(t)}\end{equation}
and the coefficient of $J\xi$ in \eqref{nn} is obtained subtracting  \eqref{second} from \eqref{prima}.
 \end{proof}
\begin{lem}\label{NXI}   At any point of the geodesic $\gamma(t)$ and for any
$v \in \gn$
\begin{equation}\label{nxi}(\na_{\xi}J)\wh v =  \sum_{i = 1}^3 a'_i \wh{J_i v}\ ,
\end{equation}
\begin{equation}\label{xin}(\na_{\wh v}J)\xi = -\frac{h}{4 f^2}\widehat{J_1v} - \frac{f'}{f}\wh{J_t v}.
\end{equation}
\end{lem}
\begin{proof} First of all, notice that   the projections of $(\na_{\xi}J)\wh v|_{\g(t)}$ and $(\na_{\wh v}J)\xi|_{\g(t)}$ along $\xi|_{\g(t)}$ or $J\xi|_{\g(t)}$ correspond to  elements of $(\gn^*)^{\Ad_K} = \{0\}$ and are therefore trivial.   Hence \eqref{nxi} is proved  observing that for any $w\in \gn$
$$ g((\na_{\xi}J)\wh v,\wh w)_{\g(t)} = (g(J\wh v,\wh w)_{\g(t)})' - g(J\wh v, \na_\xi\wh w)_{\g(t)} + g(\na_\xi \wh v,J\wh w)_{\g(t)} = $$
$$ \overset{\text{(Lemma \ref{L1})}} =  \cB\left(\sum_{i=1}^3(f^2a_i)'J_iv,w\right)  + 2\ \frac{f'}{f}g(\wh v,J\wh w)_{\g(t)} =    $$
$$ =  \frac{1}{f^2}\sum_{i=1}^3(f^2a_i)'\ g\left(\widehat{J_iv},\wh w\right)_{\g(t)}  - 2\ \frac{f'}{f}g(J \wh v,\wh w)_{\g(t)} = \sum_{i=1}^3 a_i'\ g\left(\widehat{J_iv},\wh w\right)_{\g(t)}  \ . $$
As for \eqref{xin}, we have
$$g((\na_{\wh v}J)\xi,\wh w)_{\g(t)} = g(\na_{J\xi}\wh v - J\na_{\wh v}\xi,\wh w)_{\g(t)} = u\ g(\na_{\wh A}\wh v, \wh w)_{\g(t)}  -
\frac{f'}{f}g(J\wh v,\wh w)_{\g(t)},$$
where, by Koszul's formula,
$$ g(\na_{\wh A}\wh v, \wh w)_{\g(t)} = \frac{1}{2} g([\wh A,\wh v],\wh w)_{\g(t)} - \frac{1}{2}  g([\wh w,\wh A],\wh v)_{\g(t)}- \frac{1}{2}  g([\wh w,\wh v],\wh A)_{\g(t)} = $$
$$= \frac{f^2}{2} (\cB([v,A],w) - \cB([A,w],v)) - \frac{h^2}{2}\cB([v,w],A) = $$
$$ =  -\frac{h^2}{2} \cB(w,[A,v]) = -\frac{h^2}{2 f^2} g(\wh w, \widehat{J_1v})_{\g(t)}\ .$$
Therefore $g((\na_{\wh v}J)\xi,\wh w)_{\g(t)}= -g\left(\frac{uh^2}{2f^2}\widehat{J_1v} + \frac{f'}{f}J\wh v,\wh w\right)_{\g(t)}$
and the  claim follows from the fact that $uh = 1/2$ by Lemma \ref{L1}.\end{proof}

\begin{lem}\label{NA}    At any point of the geodesic $\gamma(t)$ and for any
$v \in \gn$
\begin{equation}\label{nJxi}\left(\na_{\wh A}J \right)\wh v = \left(\frac{h^2}{2f^2} - 1\right)\widehat{[J_t,J_1]v} ,
\end{equation}
\begin{equation}\label{Jxin}(\na_{\wh v}J)\wh A = \frac{h^2}{2f^2} \widehat{J_tJ_1v} - \frac{f'}{uf}\hv.
\end{equation}
\end{lem}
\begin{proofwithnoend} By the same arguments in  the previous lemma,  $(\na_{\hA}J)\wh v|_{\g(t)}$ and
$(\na_{\wh v}J)\wh A|_{\g(t)}$
have non trivial components only along $\wh \gn$. Now,  for every $w\in \gn$, using the fact that $\wh A$ is
Killing and $J$-preserving,
$$g((\na_{\wh A}J)\wh v,\wh w) = \wh A\cdot g(J\wh v,\wh w) - g(J\wh v, \na_{\hA}\wh w) + g(\na_{\hA}\wh v,J\wh w) =$$
$$=   g(\na_{\wh v}\hA, J\wh w) - g(\na_{\wh w}\hA, J\wh v).$$
For given $t$, we set $x =  J_t w$ so that $g(\na_{\wh v}\wh A, J\wh w)_{\g(t)}Ê= g(\na_{\hv}\hA,\wh x)_{\g(t)}$. Then
$$g(\na_{\wh v}\wh A, J\wh w)_{\g(t)}Ê=   \frac{1}{2}\left( g([\hv,\hA],\wh x) + g([\hA,\wh x],\hv) -
g(\hA,[\wh x,\hv])\right)_{\g(t)} =$$
\beq \label{ok} =  \frac{1}{2} (2f^2 - h^2) \cB([A,v], J_t w) =   \left(\frac{h^2}{2 f^2} - 1\right)
g(\widehat{J_t J_1v},\hw)_{\g(t)}\ .\eeq
Similarly, one finds that
$g(\na_{\hw}\hA,J\wh v)_{\g(t)} = \left(\frac{h^2}{2 f^2} - 1\right)\ g(\widehat{J_1J_tv},\hw)_{\g(t)} $,
so that
$$g((\na_{\wh A}J)\wh v,\wh w)_{\g(t)} = \left(\frac{h^2}{2f^2} - 1\right)
g(\widehat{[J_t,J_1]v},\hw)_{\g(t)}\ ,$$
which implies \eqref{nJxi}. As for \eqref{Jxin}, the following holds at the points $\g(t)$
$$g((\na_{\hv}J)\hA,\hw) = \hv\cdot g(J\hA,\hw)- g(J\hA,\na_{\hv}\hw) - g(J\na_{\hv}\hA,\hw) = $$
$$ = - g([\hv,\hA],J\hw) - g(J\hA,\na_{\hw}\hv) + g(\na_{\hv}\hA,J\hw) = $$
$$ =  g(J \widehat{[A,v]},\hw)  + g(\hw,\na_{J\hA}\hv) + g(\na_{\hv}\hA,J\hw)
\overset{(\text{by}\ \eqref{ok} \ \text{and}\ J \wh A|_{\g(t)} = - \frac{1}{u}\xi|_{\g(t)})} = $$
$$= \frac{h^2}{2f^2}  g(\widehat{J_t J_1v},\hw)  - \frac{1}{u}  g(\hw,\na_{\xi}\hv)
\overset{\text{(Lemma \ref{L1})}}= \frac{h^2}{2f^2}  g\left(\widehat{J_t J_1v} - \frac{f'}{uf} \hv,\hw\right)\ .\eqno\qed$$
 \end{proofwithnoend}
 \bigskip
We can now insist the condition that $(\nabla_X J)X = 0$. By  $G$-invariance and  Lemmata \ref{L1} -  \ref{NA},  we obtain that $(g, J)$ is  NK if and only if  $f$, $h$ and $a_i$ are  solutions of the  equations:
\beq \label{B2} \frac{h}{4f^2} a_1 + \frac{f'}{f} = 0 \ ,\qquad
a_i' - \frac{f'}{f} a_i - \frac{h}{4 f^2}\d_{i1} = 0\ \ \ \text{for} \ i = 1,2,3\ ,\eeq
\beq \label{B3}Êa_j\left(\frac{3 h^2}{2 f^2} - 2\right) = 0 \quad \text{for}\ j = 2, 3\ ,\qquad
a_1^2 + a_2^2 + a_3^2 = 1
\eeq
(we used  the equalities  $u h = \frac{1}{2}$,  $J_1 J_2 =  J_3$ and  $J_3 J_1  = J_2$). Moreover,
\begin{lem} If the NK structure is non-K\"ahler (hence strict), then
$\frac{h}{f} =  \frac{2}{\sqrt{3}}$.
\end{lem}
\begin{proof} From $\eqref{B3}$, if $\frac{h(t_o)}{f(t_o)} \not= \frac{2}{\sqrt{3}}$ for some $t_o$, then $a_2(t )= a_3(t) = 0$ and $a'_i = 0$ on some
open neighborhood  $I_o\subset I$ of $t_o$.  By \eqref{nxi} and \eqref{nJxi}, it follows that $(\na_{\xi}J)(\wh \gn) = 0 = (\na_{\hA}J)(\wh \gn)$ and
$d\omega = 3 g(\nabla J \cdot, \cdot)$ vanishes when restricted to $\bR \xi \times \wh\gn \times \wh \gn$
and  $\bR \wh A \times \wh\gn \times \wh \gn$. As pointed out in the proof of
Lemma \ref{lemma33}, the invariance under the isotropy representation implies that $d \omega$ is identically $0$
on $\g(I_o)$, contradicting the hypotheses.  \end{proof}
Finally, observe that  when $i =2,3$ the equations $\eqref{B2}_2$   are equivalent to say that
%
there exist non zero real constants $k$, $\wt k$ such that
$$a_2 = k f\ ,\qquad a_3 = \wt k f\ .$$
By a suitable rotation  in the plane spanned by $J_2$, $J_3$, we can always assume that $\wt k = 0$. So using this and the previous lemma, the system of equation characterizing the non-K\"ahler NK structures becomes
\beq a_1 + \sqrt{12}f' = 0 \ ,\qquad
 a_2 = k f\ ,\qquad a_3 = 0\ ,\qquad h = \frac{2}{\sqrt{3}} f\ ,\eeq
\beq \label{ultima} (f')^2 - ff'' - \frac{1}{12} = 0\ ,
\qquad
12(f')^2 + k^2 f^2 = 1\ ,\eeq
which shows that any $G$-invariant NK structure is completely determined by the function $f$ and the constant $k$ that solve \eqref{ultima}. The solutions of \eqref{ultima} are of the form
\begin{equation}\label{EQ1}
 \left\{\begin{array}{cc}
f(t) = 	A \cos(\frac{k}{\sqrt{12}}t) + B \sin(\frac{k}{\sqrt{12}}t)& \\
\phantom{a}\\
(A^2 + B^2) k^2 = 1 &
\end{array}\right. \end{equation}
for some real constants $A,B$.  Under the change of  parameter
$$ u  = k t  -\arcsin(kB)\sqrt{12}$$
any function \eqref{EQ1} becomes of the form
$$f(u) = \frac{1}{k} \cos\left( \frac{u}{\sqrt{12}}\right)$$
and, at the points of $\g(t)$,  the corresponding metric $g$ is
$$g = \frac{1}{k}\left(du^2 + \cos^2(\frac{u}{\sqrt{12}})\ \cB|_{\ga\times \ga} + \frac{4}{3}\cos^2(\frac{u}{\sqrt{12}})\ \cB|_{\gn\times \gn}\right)\ .$$
This means that, up to homotheties, there is a unique  non-K\"ahler $\SU_3$-invariant NK structure on neighborhoods of regular orbits.   This uniqueness property implies that such NK structure coincides with  the standard non-K\"ahler $\SU_3$-invariant NK structure of  $S^6$ described in \S \ref{prelim}.\par

\bigskip\bigskip
\font\smallsmc = cmcsc8
\font\smalltt = cmtt8
\font\smallit = cmti8
\hbox{\parindent=0pt\parskip=0pt
\vbox{\baselineskip 9.5 pt \hsize=3.1truein
\obeylines
{\smallsmc
Fabio Podest\`a
Dip. Matematica e Appl. per l'Architettura
Universit\`a di Firenze
Piazza Ghiberti 27
50122 Firenze
ITALY}

\medskip
{\smallit E-mail}\/: {\smalltt podesta@math.unifi.it
}
}
\hskip 0.0truecm
\vbox{\baselineskip 9.5 pt \hsize=3.7truein
\obeylines
{\smallsmc
Andrea Spiro
Dip. Matematica e Informatica
Universit\`a di Camerino
Via Madonna delle Carceri
I-62032 Camerino (Macerata)
ITALY
}\medskip
{\smallit E-mail}\/: {\smalltt andrea.spiro@unicam.it}
}
}


\bigskip

\begin{thebibliography}{11}
\bibitem{A} I.~Agricola, {\it The Srn\'\i\  lectures on non-integrable geometries with torsion},
Arch. Math. (Brno) {\bf 42} (2006), 5--84.
\bibitem{AA} A.~V.~Alekseevsky, D.~V.~Alekseevsky, {\it Riemannian $G$-manifolds with one dimensional orbit space},
Ann. Glob. Anal. and Geom. {\bf 11} (1993), 197--211.
\bibitem{B} J.-B.~Butruille,  {\it Classification des vari\'{e}t\'{e}s approximativement k\"ahleriennes homog\`{e}nes}, Ann.
Glob. Anal. Geom. {\bf 27} (2005), 201–-225.
\bibitem{FI}  T.~Friedrich and S.~Ivanov,  {\it Parallel spinors and connections with skew symmetric torsion in
string theory},  Asian J. Math. {\bf 6} (2002), 303--335.
\bibitem{G3} A.~Gray, {\it Almost complex submanifolds of the six sphere}, Proc. A.M.S. {\bf 20} (1969), 277–-279.
\bibitem{G1} A.~Gray,  {\it Nearly K\"ahler manifolds}, J.Differential Geom. {\bf 4} (1970), 283--309.
\bibitem{G4} A.~Gray,  {\it Riemannian manifolds with geodesic symmetries of order $3$}, J. Diff. Geom. {\bf 7} (1972),
343--369.
\bibitem{G2} A.~Gray, {\it The structure of nearly K\"ahler manifolds}, Math. Ann. {\bf 223} (1976), 233--248.
\bibitem{GH} A.~Gray and  L.~M.~Hervella,  {\it The sixteen~classes of almost Hermitian manifolds and their local
invariants}, Ann. Mat. Pura Appl. {\bf 123} (1980), 35–-58.
\bibitem{GWZ} K.~Grove, W.~Ziller, B.~Wilking, {\it Positively curved cohomogeneity one manifolds and
3-Sasakian geometry},   J. Differential Geom. {\bf 78}  (2008), 33--111.
\bibitem{Gr} R.~Grunewald,  {\it  Six-dimensional Riemannian manifolds with a real Killing spinor},
Ann. Glob. Anal. Geom. {\bf 8} (1990), 43--59.
\bibitem{M} A.~Moroianu, P.~A.~Nagy and  U.~Semmelmann,  {\it Unit Killing vector fields on
nealy K\"ahler manifolds},  Internat. J. Math. {\bf 16}  (2005), 281--301.
\bibitem {N1} P.~A.~Nagy,  {\it Nearly K\"ahler geometry and Riemannian foliations}, Asian Math. J.
{\bf 6} (2002), 481--504.
\bibitem{N2} P.~A.~Nagy, {\it  On nearly K\"ahler geometry},  Ann. Glob. Anal. Geom. {\bf 22} (2002), 167--178.
\bibitem{U} F.~Uchida, {\it Classification of compact transformation groups on cohomology complex projective
spaces with codimension one orbits}, Japan J. Math. {\bf 3} (1977), 141-189.


\end{thebibliography}
\end{document}